\documentclass[reqno]{amsart}
\usepackage[T1]{fontenc}
\usepackage[utf8]{inputenc}
\usepackage{color}
\usepackage{url}
\usepackage{enumitem}
\usepackage{amstext}
\usepackage{amsthm}
\usepackage{amssymb}
\usepackage{esint}
\usepackage[pdfusetitle,
 bookmarks=true,bookmarksnumbered=false,bookmarksopen=false,
 breaklinks=false,pdfborder={0 0 0},pdfborderstyle={},backref=false,colorlinks=true]
 {hyperref}
\hypersetup{
 pdfborderstyle={}}

\makeatletter
\numberwithin{equation}{section}
\numberwithin{figure}{section}
\theoremstyle{plain}
\newtheorem{thm}{\protect\theoremname}
\theoremstyle{definition}
\newtheorem{defn}[thm]{\protect\definitionname}
\theoremstyle{plain}
\newtheorem{lem}[thm]{\protect\lemmaname}
\theoremstyle{definition}
\newtheorem{example}[thm]{\protect\examplename}

\usepackage{pdfsync}
\usepackage{graphics}
\usepackage{epstopdf}
\usepackage[all]{xy}
\usepackage{amscd}
\usepackage{mathtools}
\usepackage{stmaryrd}
\usepackage[utf8]{inputenc}
\setlist[enumerate]{leftmargin=*,label=(\roman*),align=left}

\makeatletter
\@namedef{subjclassname@2020}{%
  \textup{2020} Mathematics Subject Classification}
\makeatother

\newcommand{\xyR}[1]{ \makeatletter
\xydef@\xymatrixrowsep@{#1} \makeatother} 
\newcommand{\xyC}[1]{ \makeatletter
\xydef@\xymatrixcolsep@{#1} \makeatother} 
\entrymodifiers={++[ ][F]} 

\newcommand{\ra}{\longrightarrow}

\newcommand{\field}[1]{\mathbb{#1}}
\newcommand{\R}{\field{R}} 
\newcommand{\N}{\field{N}} 

\newcommand{\eps}{\varepsilon} 
\renewcommand{\phi}{\varphi}
\newcommand{\diff}[1]{\ifmmode\mathchoice{\hbox{\rm d}#1}  
 {\hbox{\rm d}#1}  
 {\scalebox{0.75}{$\hbox{\rm d}#1$}}  
 {\scalebox{0.35}{$\hbox{\rm d}#1$}}  
 \fi} 

\newcommand{\abs}[2][\empty]{\ifx#1\empty\left|#2\right|%
\else#1\vert #2 #1\vert\fi}

\newcommand{\cinfty}{\mbox{\ensuremath{\mathcal{C}}}^{\infty}}

\newcommand{\Rtil}{\widetilde \R} 

\newcommand{\frontRise}[2]{\ifmmode\mathchoice{{\vphantom{#1}}^{\scalebox{0.6}{$#2$}}}  
 {{\vphantom{#1}}^{\scalebox{0.56}{$#2$}}}  
 {{\vphantom{#1}}^{\scalebox{0.47}{$#2$}}}  
 {{\vphantom{#1}}^{\scalebox{0.35}{$#2$}}}\fi} 
\newcommand{\RC}[1]{\frontRise{\R}{#1}\Rtil}
\newcommand{\rcrho}{\RC{\rho}}
\newcommand{\rti}{\RC{\rho}}
\newcommand{\gsf}{\frontRise{\mathcal{G}}{\rho}\mathcal{GC}^{\infty}}
\newcommand{\gsfk}[1]{\frontRise{\mathcal{G}}{\rho}\mathcal{GC}^{#1}}

\newcommand{\frontRiseDown}[3]{\ifmmode\mathchoice{{\vphantom{#1}}^{\scalebox{0.6}{$#2$}}_{\scalebox{0.6}{$#3$}}}  
 {{\vphantom{#1}}^{\scalebox{0.56}{$#2$}}_{\scalebox{0.56}{$#3$}}}  
 {{\vphantom{#1}}^{\scalebox{0.47}{$#2$}}_{\scalebox{0.47}{$#3$}}}  
 {{\vphantom{#1}}^{\scalebox{0.35}{$#2$}}_{\scalebox{0.35}{$#3$}}}\fi} 

\newcommand{\ptind}{\displaystyle \mathop {\ldots\ldots\,}} 

\newcommand{\no}[1]{| #1|}
\newcommand{\Gcinf}{\mathcal{G}\cinfty}
\newcommand{\Om}{\Omega}
\newcommand{\st}[1]{{#1^\circ}}
\newcommand{\gs}{\mathcal{G}^s}

\newcommand{\sse}{\subseteq}

\makeatother

\providecommand{\definitionname}{Definition}
\providecommand{\examplename}{Example}
\providecommand{\lemmaname}{Lemma}
\providecommand{\theoremname}{Theorem}

\begin{document}
\title{Classical finite dimensional fixed point methods for generalized functions}
\author{Kevin Islami \and George Apaaboah \and Paolo Giordano}
\thanks{This research was also funded in whole or in part by the Austrian
Science Fund (FWF) 10.55776/PAT9221023, P34113, 10.55776/P33945, 10.55776/P33538.
For open access purposes, the author has applied a CC BY 4.0 public
copyright license to any author-accepted manuscript version arising
from this submission.}
\address{\textsc{Faculty of Mathematics, University of Vienna, Austria}}
\email{kevin.islami@univie.ac.at}
\address{\textsc{Siren} and \textsc{University Grenoble Alpes, France}}
\email{george.apaaboah@gmail.com}
\address{\textsc{Faculty of Mathematics, University of Vienna, Austria}}
\email{paolo.giordano@univie.ac.at}
\begin{abstract}
We prove Banach, Newton-Raphson and Brouwer fixed point theorems in
the framework of generalized smooth functions, a minimal extension
of Colombeau's theory (and hence of classical distribution theory)
which makes it possible to model nonlinear singular problems, while
at the same time sharing a number of fundamental properties with ordinary
smooth functions, such as the closure with respect to composition
and several non trivial classical theorems of the calculus. The proved
results allows one to deal with equations of the form $F(x)=0$, where
$F$ is a generalized smooth function, in particular, a Sobolev-Schwartz
distribution. We consider examples with singularities that are not
included in the classical version of these theorems.
\end{abstract}

\maketitle

\section{Introduction: solving nonlinear equations with generalized functions}

\noindent An superficial assumption of Leibniz's viewpoint \emph{natura
non facit saltus}, \cite{Leib}, immediately collapses with our present
mathematical models, which frequently consider various types of discontinuities.
For example, we could be interested in large deformations of solid
bodies or very rapid mechanical fractures (elastoplasticity), \cite{UgFe11,ChCh11,C1},
ruptures in geophysics, \cite{BrHoHo17,HoKuYe19}, micromechanical
models for granular materials, \cite{ChMeNe81,BoBo17}, in the dynamics
of two interpenetrating fluids (multifluid flows), \cite{StWe84,LuTr19},
discontinuities in Lagrangian optics, \cite{LaGhTh11}, wave propagation
in a medium with piecewise smooth characteristics (e.g.~seismic waves
in stratified media, \cite{Ken09}), or solutions of the Schrödinger
equation for an infinite rectangular potential well or even a rectangular
potential well with periodic barriers changing at infinite frequency,
\cite{GaPa90,VKRKPW}. A common mathematical framework used in these
models is that of Sobolev spaces, and hence of weak derivatives, and
Sobolev-Schwartz distributions. However, the informal use of nonlinear
operations, pointwise values for generalized functions and their derivatives,
and also the useful adoption of infinitesimal or infinite quantities
to construct idealized models, are necessary tools in dealing with
singular models that cannot be formalized using the aforementioned
mathematical theories (see e.g.~\cite{C1,BIG} for a general overview
and the previous citations for specific applications). Clearly, opting
\emph{only} for numerical solutions does not allow to have general
theorems studying these solutions and their properties, and hence
a broader and deeper understanding.

J.F.~Colombeau's theory of generalized functions enables all these
modeling possibilities between embedded Sobolev-Schwartz distributions,
avoiding the difficulty of the Schwartz impossibility theorem. See
e.g.\ \cite{C1,Obe92,Pil94,GKOS} for an introduction with applications.
In particular, generalized smooth functions (GSF) are a minimal extension
of Colombeau's theory (and hence of classical distribution theory)
which makes it possible to model nonlinear singular problems, while
at the same time sharing a number of fundamental properties with ordinary
smooth functions, such as the closure with respect to composition
and several non trivial classical theorems of the calculus, see \cite{BIG,GGBL,MTG,LeLuBaGi17,GKV19}.
One could describe GSF as a methodological restoration of Cauchy-Dirac's
original conception of generalized function, see \cite{Lau89,KaTa12}.
In essence, the idea of Cauchy and Dirac (but also of Poisson, Kirchhoff,
Helmholtz, Kelvin and Heaviside) was to view generalized functions
as suitable types of smooth set-theoretical maps, obtained from ordinary
smooth maps depending on suitable infinitesimal or infinite parameters.

In the present paper, we deal with Banach, Newton and Brouwer fixed
point methods in the framework of GSF. This allows us to solve equations
of the form $F(x)=0$, where $F$ is a GSF, in particular, a Sobolev-Schwartz
distribution. Even if we consider only the case where $x$ is a (non-Archimedean,
i.e.~$x$ can also be an infinitesimal or an infinite number) $n$-dimensional
point, this work is preliminary to a subsequent one where we will
deal with infinite-dimensional equations.

The paper is self-contained in the sense that it contains all the
statements required for the proofs we are going to present, all the
needed notions of GSF theory are introduced, and only an elementary
knowledge of Colombeau theory is needed.

In Sec.~\ref{sec:Preliminary-notions}, we provide an introduction
to generalized numbers and their topology, as well as GSF and some
of their analytical properties. Sec.~\ref{sec:Banach-fixed-point}
treats the Banach fixed point theorem in finite dimensional spaces
of generalized points. We continue in Sec.~\ref{sec:Newton's-method}
to treat Newton's method in finite dimensions and the corresponding
quadratic convergence. Then, we conclude this work with some examples
in Sec.~\ref{sec:Examples}, in particular we consider examples with
singularities that are not included in the classical version of these
methods.

\section{\label{sec:Preliminary-notions}Preliminary notions}

\subsection{\label{subsec:The-new-ring}The ring of scalars}

In this work, $I$ denotes the interval $(0,1]\subseteq\R$ and we
will always use the variable $\eps$ for elements of $I$. $\eps$-dependent
nets $x\in\R^{I}$ are denoted by $(x_{\eps})$ and $\N$ denotes
the set of natural numbers, including zero.

\noindent We start by defining the new simple non-Archimedean ring
of scalars $\rti$ that generalizes the Colombeau ring of generalized
numbers $\Rtil$ when $\rho=(\rho_{\eps})=(\eps)$. For all the proofs
of results in this section, see \cite{GKV19}.
\begin{defn}
\label{def:RCGN}Let $\rho=(\rho_{\eps})\in(0,1]^{I}$ be a net such
that $\lim_{\eps\to0^{+}}\rho_{\eps}=0$, then
\begin{enumerate}
\item $\mathcal{I}(\rho):=\left\{ (\rho_{\eps}^{-a})\mid a\in\R_{>0}\right\} $
is called the \emph{asymptotic gauge} generated by $\rho$. The net
$\rho$ is called a \emph{gauge}.
\item If $\mathcal{P}(\eps)$ is a property of $\eps\in I$, we use the
notation $\forall^{0}\eps:\,\mathcal{P}(\eps)$ to denote $\exists\eps_{0}\in I\,\forall\eps\in(0,\eps_{0}]:\,\mathcal{P}(\eps)$.
We read $\forall^{0}\eps$ as: \emph{``for $\eps$ small}''.
\item We say that a net $(x_{\eps})\in\R^{I}$ \emph{is $\rho$-moderate},
and we write $(x_{\eps})\in\R_{\rho}$ if 
\[
\exists(J_{\eps})\in\mathcal{I}(\rho):\ x_{\eps}=O(J_{\eps})\text{ as }\eps\to0^{+}
\]
i.e., if $\exists N\in\mathbb{N}\,\forall^{0}\eps:\ \lvert x_{\eps}\rvert\leq\rho_{\eps}^{-N}$.
\end{enumerate}
Let $(x_{\eps})$, $(y_{\eps})\in\R^{I}$, then we say that $(x_{\eps})\sim_{\rho}(y_{\eps})$
if 
\[
\forall(J_{\eps})\in\mathcal{I}(\rho):\ x_{\eps}=y_{\eps}+O(J_{\eps}^{-1})\text{ as }\eps\to0^{+},
\]
that is if $\forall n\in\N\,\forall\eps:\ |x_{\eps}-y_{\eps}|\le\rho_{\eps}^{n}$.
This is a congruence relation on the ring $\R_{\rho}$ of moderate
nets with respect to pointwise operations, and we can hence define
\[
\RC{\rho}:=\R_{\rho}/\sim_{\rho},
\]
which we call \emph{Robinson-Colombeau ring of generalized numbers},
\cite{Rob73,C1}. We denote the equivalence class $x\in\rti$ by $x=:[x_{\eps}]:=[(x_{\eps})]_{\sim_{\rho}}\in\rti$.
\end{defn}

In the following, $\rho$ will always denote an infinitesimal net
as in Def.~\ref{def:RCGN}, and it can be chosen, e.g., depending
on the class of (differential) equations that needs to be solved for
the generalized functions we are going to introduce, see \cite{GiLu16}.
For motivations concerning the naturality of $\rti$, see \cite{GKV19,KeGi24a}.
The case $\rho_{\eps}=\eps$ corresponds to the ring $\Rtil$ of Colombeau
generalized numbers \cite{GKOS}; for basic properties of $\rti$,
see e.g.~\cite{GKV19}.

\noindent We can define an order relation on $\RC{\rho}$ by saying
that $[x_{\eps}]\le[y_{\eps}]$, if there exists $(z_{\eps})\in\R^{I}$
such that $(z_{\eps})\sim_{\rho}0$ (we say that $(z_{\eps})$ is
\emph{$\rho$-negligible}) and $x_{\eps}\le y_{\eps}+z_{\eps}$ for
$\eps$ small. We say that $x<y$ if $x\le y$ and $x-y$ is invertible.
A proficient intuitive point of view on these generalized numbers
is to think of $[x_{\eps}]\in\rcrho$ as a \emph{dynamic point in
the time $\eps\to0^{+}$}; classical real numbers are hence embedded
as \emph{static points}. Furthermore, we say that $x=[x_{\eps}]\in\rti$
is \emph{near-standard} if $\exists\lim_{\eps\to0^{+}}x_{\eps}=:\st{x}\in\R$.

\noindent Although the order $\le$ is not total, we still have the
possibility to define the infimum $\min\left([x_{\eps}],[y_{\eps}]\right):=[\min(x_{\eps},y_{\eps})]$,
and analogously the supremum function $\max\left([x_{\eps}],[y_{\eps}]\right):=\left[\max(x_{\eps},y_{\eps})\right]$
of a finite number of points, and the absolute value $|[x_{\eps}]|:=[|x_{\eps}|]\in\RC{\rho}$
as well. In the following, we will also use the customary notation
$\RC{\rho}^{*}$ for the set of invertible generalized numbers. Our
notations for intervals are: $[a,b]:=\{x\in\RC{\rho}\mid a\le x\le b\}$,
$[a,b]_{\R}:=[a,b]\cap\R$, and analogously for segments $[x,y]:=\left\{ x+r\cdot(y-x)\mid r\in[0,1]\right\} \subseteq\RC{\rho}^{n}$
and $[x,y]_{\R^{n}}=[x,y]\cap\R^{n}$. Open intervals are defined
using the relation $<$, i.e.~$(a,b):=\{x\in\rcrho\mid a<x<b\}$.
Finally, we set $\diff{\rho}:=[\rho_{\eps}]\in\RC{\rho}$, which is
a positive invertible infinitesimal, whose reciprocal is $\diff{\rho}^{-1}=[\rho_{\eps}^{-1}]$
and which is necessarily a positive infinite number.

As in every non-Archimedean ring, we have the following
\begin{defn}
\label{def:nonArchNumbs}Let $x\in\RC{\rho}$ be a generalized number,
then
\begin{enumerate}
\item $x$ is \emph{infinitesimal} if $|x|\le r$ for all $r\in\R_{>0}$.
If $x=[x_{\eps}]$, this is equivalent to $\lim_{\eps\to0^{+}}x_{\eps}=0$.
This intuitively clear result is neither possible in nonstandard analysis,
\cite{Cut}, nor in synthetic differential geometry, \cite{Lav}.
We write $x\approx y$ if $x-y$ is infinitesimal, and write $D_{\infty}:=\left\{ h\in\rti\mid h\approx0\right\} $
for the set of all infinitesimals.
\item $x$ is \emph{infinite} if $|x|\ge r$ for all $r\in\R_{>0}$. If
$x=[x_{\eps}]$, this is equivalent to $\lim_{\eps\to0^{+}}\left|x_{\eps}\right|=+\infty$.
\item $x$ is \emph{finite} if $|x|\le r$ for some $r\in\R_{>0}$.
\end{enumerate}
\end{defn}

\subsection{\label{subsec:Topologies}Topologies on $\RC{\rho}^{n}$}

On the $\RC{\rho}$-module $\RC{\rho}^{n}$, we can consider the natural
extension of the Euclidean norm, i.e.~$|[x_{\eps}]|:=[|x_{\eps}|]\in\RC{\rho}$,
where $[x_{\eps}]\in\RC{\rho}^{n}$. Even if this generalized norm
takes values in $\RC{\rho}$, it shares the common properties with
classical norms, see e.g.~\cite{GKV19}:
\begin{enumerate}
\item $\lvert x\vert=\max(x,-x)$
\item $\lvert x\rvert\geq0$
\item $\lvert x\rvert=0\implies x=0$
\item $\lvert y\cdot x\rvert=\lvert y\rvert\cdot\lvert x\rvert$
\item $\lvert x+y\rvert\leq\lvert x\rvert+\lvert y\rvert$
\item $\lvert\lvert x\rvert-\lvert y\rvert\rvert\leq\lvert x-y\rvert$.
\end{enumerate}
It is therefore natural to consider, on $\RC{\rho}^{n},$ topologies
generated by balls defined by this generalized norm and a set of radii.
It is a non-trivial step to understand that a meaningful set of radii
that allows one to get continuity of our generalized functions is
the set $\RC{\rho}_{\ge0}^{*}=\rcrho_{>0}$ of positive and invertible
generalized numbers. In fact, the set of balls $\left\{ B_{r}(c)\mid r\in\rti_{>0},\ c\in\RC{\rho}^{n}\right\} $,
where $B_{r}(c):=\left\{ x\in\rti^{n}\mid|x-c|<r\right\} $, is a
base for a topology on $\RC{\rho}^{n}$, called \emph{sharp topology},
and any open set in the sharp topology will be called \emph{sharply
open set}, see \cite{ArJu,Sca00,Sca98,Ob-Ve,GKOS,GKV15} and references
therein. The following result is useful when dealing with positive
and invertible generalized numbers (cf.~\cite{GKOS}).
\begin{lem}
\label{lem:mayer} Let $x\in\RC{\rho}$. Then the following are equivalent:
\begin{enumerate}
\item \label{enu:positiveInvertible}$x$ is invertible and $x\ge0$, i.e.~$x>0$.
\item \label{enu:strictlyPositive}For each representative $(x_{\eps})\in\R_{\rho}$
of $x$ we have $\forall^{0}\eps:\ x_{\eps}>0$.
\item \label{enu:greater-i_epsTom}For each representative $(x_{\eps})\in\R_{\rho}$
of $x$ we have $\exists m\in\N\,\forall^{0}\eps:\ x_{\eps}>\rho_{\eps}^{m}$.
\item There exists a representative $(x_{\eps})\in\text{\ensuremath{\mathbb{R}_{\rho}}}$
of $x$ such that $\exists m\in\mathbb{N}\,\forall^{0}\eps:\ x_{\eps}>\rho_{\eps}^{m}$.
\end{enumerate}
\end{lem}

\noindent The ring $\rcrho$ is Cauchy complete and Hausdorff in the
sharp topology, \cite{GK15,GKV19}, and from Lem.~\ref{lem:mayer},
it also follows that the sharp topology is generated by all the infinitesimal
balls of the type $B_{\diff\rho^{q}}(c)\subseteq\rcrho^{n}$. Therefore,
a sequence $(x_{n})_{n\in\N}$ of $\rcrho^{n}$ converges to its limit
$l$ in the sharp topology if and only if
\begin{equation}
\forall q\in\rcrho_{>0}\,\exists N\in\mathbb{N}\,\forall n\in\N_{>N}:\ \lvert x_{n}-l\rvert<\diff\rho^{q}.\label{eq:limit}
\end{equation}

\subsection*{Generalized smooth functions and their differential calculus}

Using the ring $\rti$, it is easy to consider a Gaussian with an
infinitesimal standard deviation. If we denote this probability density
by $f(x,\sigma)$, and set $\sigma=[\sigma_{\eps}]\in\RC{\rho}_{>0}$,
where $\sigma\approx0$, we obtain a net of smooth functions $(f(-,\sigma_{\eps}))_{\eps\in I}$.
This is the basic idea behind the following
\begin{defn}
\label{def:netDefMap}Let $X\subseteq\RC{\rho}^{n}$ and $Y\subseteq\RC{\rho}^{d}$
be arbitrary subsets of generalized points. Then, we say that 
\[
f:X\longrightarrow Y\text{ is a \emph{generalized smooth function}}
\]
if $f:X\ra Y$ is a set theoretical map and there exists a net $f_{\eps}\in\cinfty(\R^{n},\R^{d})$
defining $f$ in the sense that $f([x_{\eps}])=[f_{\eps}(x_{\eps})]\in Y$
and $(\partial^{\alpha}f_{\eps}(x_{\eps}))\in\R_{{\scriptscriptstyle \rho}}^{d}$
for all representatives $[x_{\eps}]=x\in X$ of any point in the domain
$X$ and all derivatives $\alpha\in\N^{n}$. The space of generalized
smooth functions (GSF) mapping $X$ into $Y$ is denoted by $\gsf(X,Y)$.
\end{defn}

\noindent Let us note explicitly that this definition states minimal
logical conditions to obtain a set-theoretical map from $X$ into
$Y$ defined by a net of smooth functions, see \cite[Sec.~5]{KeGi24a}.
In particular, the following Thm.~\ref{thm:propGSF} states that
the equality $f([x_{\eps}])=[f_{\eps}(x_{\eps})]$ is well-defined,
i.e.~that we have independence from the representatives for all derivatives
$[x_{\eps}]\in X\mapsto[\partial^{\alpha}f_{\eps}(x_{\eps})]\in\RC{\rho}^{d}$,
$\alpha\in\N^{n}$.
\begin{thm}
\label{thm:propGSF}Let $X\subseteq\RC{\rho}^{n}$ and $Y\subseteq\RC{\rho}^{d}$
be arbitrary subsets of generalized points. Let $f_{\eps}\in\cinfty(\R^{n},\R^{d})$
be a net of smooth functions that defines a generalized smooth map
of the type $X\longrightarrow Y$, then
\begin{enumerate}
\item $\forall\alpha\in\N^{n}\,\forall(x_{\eps}),(x'_{\eps})\in\R_{\rho}^{n}:\ [x_{\eps}]=[x'_{\eps}]\in X\ \Rightarrow\ (\partial^{\alpha}f_{\eps}(x_{\eps}))\sim_{\rho}(\partial^{\alpha}f_{\eps}(x'_{\eps}))$.
\item \label{enu:GSF-cont}Each $f\in\gsf(X,Y)$ is continuous with respect
to the sharp topologies induced on $X$, $Y$.
\item $f:X\longrightarrow Y$ is a GSF if and only if there exists a net
$v_{\eps}\in\cinfty(\Omega_{\eps},\R^{d})$ such that $\left\langle \Omega_{\eps}\right\rangle :=\left\{ x\in\rti^{n}\mid\forall[x_{\eps}]=x\,\forall^{0}\eps:\ x_{\eps}\in\Omega_{\eps}\right\} \supseteq X$,
$(\partial^{\alpha}v_{\eps}(x_{\eps}))\in\R_{{\scriptscriptstyle \rho}}^{d}$
for all representatives $[x_{\eps}]=x\in X$ and all $\alpha\in\N^{n}$,
and $f=[v_{\eps}(-)]|_{X}$.
\end{enumerate}
\end{thm}

The differential calculus for GSF can be introduced by showing existence
and uniqueness of another GSF serving as an incremental ratio.
\begin{thm}[Fermat-Reyes theorem for GSF]
\label{thm:FR-forGSF} Let $U\subseteq\RC{\rho}^{n}$ be a sharply
open set, let $v=[v_{\eps}]\in\RC{\rho}^{n}$, and let $f\in\gsf(U,\RC{\rho})$
be a generalized smooth map generated by the net of smooth functions
$f_{\eps}\in\cinfty(\Omega_{\eps},\R)$. Then
\begin{enumerate}
\item \label{enu:existenceRatio}There exists a sharp neighborhood $T$
of $U\times\{0\}$ and a generalized smooth map $r\in\gsf(T,\RC{\rho})$,
called the \emph{generalized incremental ratio} of $f$ \emph{along}
$v$, such that 
\[
\forall(x,h)\in T:\ f(x+hv)=f(x)+h\cdot r(x,h).
\]
\item \label{enu:uniquenessRatio}Any two generalized incremental ratios
coincide on a sharp neighborhood of $U\times\{0\}$.
\item \label{enu:defDer}We have $r(x,0)=\left[\frac{\partial f_{\eps}}{\partial v_{\eps}}(x_{\eps})\right]$
for every $x\in U$ and we can thus define $\diff f(x)\cdot v:=\frac{\partial f}{\partial v}(x):=r(x,0)$,
so that $\frac{\partial f}{\partial v}\in\gsf(U,\RC{\rho})$.
\end{enumerate}
\end{thm}

Note that this result permits considering the partial derivative of
$f$ with respect to an arbitrary generalized vector $v\in\RC{\rho}^{n}$
which can be, e.g., infinitesimal or infinite. Using this result,
we can also define subsequent differentials $\diff{}^{j}f(x)$ as
$j-$multilinear maps, and we set $\diff{}^{j}f(x)\cdot h^{j}:=\diff{}^{j}f(x)(h,\ptind^{j},h)$.
The set of all $j-$multilinear maps $\left(\rti^{n}\right)^{j}\ra\rti^{d}$
over the ring $\rti$ will be denoted by $L^{j}(\rti^{n},\rti^{d})$.
For $A=[A_{\eps}(-)]\in L^{j}(\rti^{n},\rti^{d})$, we set $\no{A}:=[\no{A_{\eps}}]$,
the generalized number defined by the operator norms of the multilinear
maps $A_{\eps}\in L^{j}(\R^{n},\R^{d})$.

In contrast to the case of distributions and Colombeau generalized
functions, there is no problem in considering the free composition
of two GSF. This property opens up new interesting possibilities,
e.g.,~in considering differential equations $y'=f(y,t)$, where $y$
and $f$ are GSF. For instance, there is no problem in studying $y'=\delta(y)$.
See \cite{GKV19} for the proof of the following
\begin{thm}
\label{thm:GSFcategory} Subsets $S\subseteq\RC{\rho}^{s}$ with the
trace of the sharp topology, and generalized smooth maps as arrows
form a subcategory of the category of topological spaces. We will
call this category $\gsf$, the \emph{category of GSF}.
\end{thm}

\noindent For instance, we can think of the Dirac delta as a map of
the form $\delta:\RC{\rho}\longrightarrow\RC{\rho}$, and therefore
the composition $e^{\delta}$ is defined in $\{x\in\RC{\rho}\mid\exists z\in\RC{\rho}_{>0}:\ \delta(x)\le\log z\}$,
which of course does not contain $x=0$ but only suitable non zero
infinitesimals. On the other hand, $\delta\circ\delta:\RC{\rho}\ra\RC{\rho}$.

The usual rules of differential calculus hold, chain rule included:
$\frac{\partial\left(f\circ g\right)}{\partial v}(x)=\diff{f}\left(g(x)\right).\frac{\partial g}{\partial v}(x)$
if $U\subseteq\rcrho^{n}$ and $V\subseteq\rcrho^{d}$ are open subsets
in the sharp topology and $g\in{}^{\rho}\Gcinf(V,U)$, $f\in{}^{\rho}\Gcinf(U,\rcrho)$.

We also have a generalization of the Taylor formula:
\begin{thm}
\label{thm:Taylor}Let $f\in\gsf(U,\rcrho)$ be a generalized smooth
function defined in the sharply open set $U\subseteq\rcrho^{d}$.
Let $a$, $b\in\rcrho^{d}$ such that the line segment $[a,b]\subseteq U$,
and set $h:=b-a$. Then, for all $n\in\N$ we have
\begin{enumerate}
\item \label{enu:LagrangeRem}$\exists\xi\in[a,b]:\ f(a+h)=\sum_{j=0}^{n}\frac{\diff{^{j}f}(a)}{j!}\cdot h^{j}+\frac{\diff{^{n+1}f}(\xi)}{(n+1)!}\cdot h^{n+1}.$
\item \label{enu:integralRem}$f(a+h)=\sum_{j=0}^{n}\frac{\diff{^{j}f}(a)}{j!}\cdot h^{j}+\frac{1}{n!}\cdot\int_{0}^{1}(1-t)^{n}\,\diff{^{n+1}f}(a+th)\cdot h^{n+1}\,\diff{t}.$
\end{enumerate}
\noindent Moreover, there exists some $R\in\rcrho_{>0}$ such that
\begin{equation}
\forall k\in B_{R}(0)\,\exists\xi\in[a,a+k]:\ f(a+k)=\sum_{j=0}^{n}\frac{\diff{^{j}f}(a)}{j!}\cdot k^{j}+\frac{\diff{^{n+1}f}(\xi)}{(n+1)!}\cdot k^{n+1}\label{eq:LagrangeInfRest-1-1}
\end{equation}
\begin{equation}
\frac{\diff{^{n+1}f}(\xi)}{(n+1)!}\cdot k^{n+1}=\frac{1}{n!}\cdot\int_{0}^{1}(1-t)^{n}\,\diff{^{n+1}f}(a+tk)\cdot k^{n+1}\,\diff{t}\approx0.\label{eq:integralInfRest-1-1}
\end{equation}
\end{thm}

Formula \ref{enu:LagrangeRem} corresponds to a direct generalization
of Taylor formulas for ordinary smooth functions with Lagrange remainder.
On the other hand, in \eqref{eq:LagrangeInfRest-1-1} and \eqref{eq:integralInfRest-1-1},
the possibility of the differential $\diff{}^{n+1}f$ being infinite
at some point is considered, and the Taylor formulas are stated so
as to have infinitesimal remainder.

For the embedding of Schwartz distributions, one can proceed exactly
as in Colombeau theory, so we refer to \cite{GKOS,GKV19}. It follows
from \cite[Th.~37]{GKOS} that the special Colombeau algebra $\gs(\Om)$
can be identified, in the special case where $\rho(\eps)=\eps$, with
the algebra $\gsf(\widetilde{\Omega}_{\text{c}},\RC{\rho})$ of GSF
defined on compactly supported points $\widetilde{\Omega}_{\text{c}}:=\{x\in\rti\mid x\text{ is finite},\ \exists r\in\R_{>0}:\ d(z,\partial\Omega)\ge r\}$
of $\Omega$. Therefore, GSF can have more general domains with respect
to Colombeau generalized functions, and this results in the closure
with respect to composition.

\section{\label{sec:Banach-fixed-point}Banach fixed point theorem in $n$-dimensions}

For the sake of clarity, in the present work we start by considering
finite dimensional fixed point theorems with only a \emph{finite}
number of iterations $f^{n}=f\circ\ptind^{n}\circ f$, $n\in\N$.
Only subsequently, in \cite{KeGi24b}, we will consider the case of
a \emph{hyperfinite} $n=[n_{\eps}]\in\rti$, $n_{\eps}\in\N$, number
$f^{n}(x)=[(f_{\eps}\circ\ptind^{n_{\eps}}\circ f_{\eps})(x_{\eps})]$
of iterations.

\subsection{\label{subsec:Finite-Banach-fixed}Banach fixed point theorem in
$d$-dimensions}

Our proof of the Kantorovich theorem for the convergence of Newton-Raphson
process is based on \cite[Thm.~1]{BIA}. The main idea is to prove
this convergence using the Banach fixed point theorem. However, this
process starts from a point $x_{0}\in X$ and recursively defines
only a sequence and not a contraction on the entire subset $X$, see
again \cite[Thm.~1]{BIA}. For this reason, we need the following
\begin{defn}
\label{def:contraction}Let $X\subseteq\rcrho^{d}$ be a sharply closed
subset and $x_{0}\in X$. Then, we say that $g$ is a \emph{contraction
on the orbit starting from $x_{0}\in X$} if:
\begin{enumerate}
\item $g\in\gsf(X,\rti^{d})$;
\item \label{enu:g^nInX}$g^{n}(x_{0})\in X$ for all $n\in\N$;
\item \label{enu:inequalityContraction}$\exists\alpha\in\rcrho_{>0}\,\forall n\in\N:\ |g^{n+1}(x_{0})-g^{n}(x_{0})|\le\alpha^{n}|g(x_{0})-x_{0}|$;
\item \label{enu:limitContraction}$\forall n\in\mathbb{N}:\ \lim_{\substack{n\to+\infty\\
n\in\N
}
}\alpha^{n}=0$, where the limit is taken in the sharp topology, see \eqref{eq:limit}.
\end{enumerate}
Moreover, we say that $g$ is a \emph{contraction on $X$} if $g:X\ra X$
and $|g(x)-g(y)|\le\alpha|x-y|$ for all $x$, $y\in X$. In both
cases, such an $\alpha\in\rcrho$ is called a \emph{contraction constant}
for $g$.
\end{defn}

\noindent We immediately note that \ref{enu:limitContraction} is
one of the peculiar differences with respect to the classical definition
of a contraction in a Banach space. It is a stronger condition because
it entails
\begin{equation}
\forall q\in\N\,\exists N\in\N_{>0}\,\forall n\in\N_{\ge N}:\ \alpha<\diff{\rho}^{\frac{q}{n}}.\label{eq:strongInfinitesimal}
\end{equation}
Therefore, $\alpha$ is necessarily an infinitesimal number; e.g.~any
$0<\alpha\le\diff{\rho}^{k}$, $k\in\N_{>0}$, satisfies \ref{enu:limitContraction}.
Thus a contraction is a generalized function that maps points closer
together and with a strong infinitesimal contraction constant. By
the first order Taylor formula (mean value theorem) with Lagrange
remainder Thm.~\ref{thm:Taylor}.\ref{enu:LagrangeRem}, if $\left|g'(x)\right|\le\alpha$
for all $x\in X$ and $\alpha$ satisfies \eqref{eq:strongInfinitesimal},
then $\alpha$ is a contraction constant for $g$. To overcome this
limitation, we need to consider a hyperfinite $n=[n_{\eps}]\in\rti$,
$n_{\eps}\in\N$, number of iterations. This is considered in the
next work \cite{KeGi24b}.

Thanks to all the preliminary results we have already proved for GSF
(of particular importance from this viewpoint is the closure with
respect to composition, Thm.~\ref{thm:GSFcategory}) the proof of
the Banach fixed theorem is formally similar to the classical one.
\begin{lem}
\label{lem:atMost1}In the assumptions of Def.~\ref{def:contraction},
we have:
\begin{enumerate}
\item A contraction on $X$ is a contraction on the orbit starting from
any point $x_{0}\in X$;
\item A contraction on $X$ can have at most one fixed point.
\end{enumerate}
\end{lem}

\begin{proof}
The first property simply follows by proving by induction the inequality
\ref{enu:inequalityContraction} of Def.~\ref{def:contraction}.
For the second claim, let us assume that $x$, $y\in X\in\rcrho$
are fixed points of $g$, so that $g(x)=x$ and $g(y)=y$. Since $\alpha<1$,
as seen in \eqref{eq:strongInfinitesimal} and by \ref{enu:inequalityContraction}
in Def.~\ref{def:contraction}, we have that $\left|g(x)-g(y)\right|=|x-y|\le|x-y|$,
i.e.~$|x-y|=0$ and hence $x=y$.
\end{proof}
\begin{thm}[Banach]
\label{thm:BFPT}Let $x_{0}\in X\subseteq\rcrho^{d}$ be a sharply
closed subset. Let $g:X\longrightarrow\rti^{d}$ be a contraction
on the orbit starting from $x_{0}$. Then $\{g^{n}(x_{0})\}_{n\in\mathbb{N}}$
is a Cauchy sequence in the sharp topology and its limit $x^{*}\in X$
is a fixed point of $g$.
\end{thm}

\begin{proof}
Using the contraction property \ref{enu:inequalityContraction} of
Def.~\ref{def:contraction}, for all $n$, $m\in\mathbb{N}$, $n<m$,
we have:
\begin{align*}
|g^{m}(x_{0})-g^{n}(x_{0})| & \le|g^{m}(x_{0})-g^{m-1}(x_{0})|+\dots+|g^{n+1}(x_{0})-g^{n}(x_{0})|\\
 & \le\alpha^{n}\left(\sum_{j=0}^{m-n-1}\alpha^{j}\right)|g(x_{0})-x_{0}|\\
 & =\alpha^{n}\left(\frac{1-\alpha^{m-n}}{1-\alpha}\right)|g(x_{0})-x_{0}|\\
 & =\frac{\alpha^{n}-\alpha^{m}}{1-\alpha}|g(x_{0})-x_{0}|
\end{align*}
which implies that $\{g^{n}(x_{0})\}_{n\in\mathbb{N}}$ is a Cauchy
sequence of $X$ because of \ref{enu:g^nInX} and \ref{enu:limitContraction}
of Def.~\ref{def:contraction}. Since $X$ is complete in the sharp
topology, the sequence $\{g^{n}(x_{0})\}_{n\in\mathbb{N}}$ has a
limit $x^{*}\in X$. By the continuity of $g$ in the sharp topology
(Thm. \ref{thm:propGSF}.\ref{enu:GSF-cont}), we have the conclusion:

\[
g(x^{*})=g\Big(\lim_{n\in\mathbb{N}}g^{n}(x_{0})\Big)=\lim_{n\in\mathbb{N}}g^{n+1}(x_{0})=x^{*}.
\]
\end{proof}
\noindent For Banach fixed point theorem in the setting of Colombeau
algebra, see \cite{Ma1,JuOl1}.

To finish this section, we will state and prove Brouwer's fixed point
Theorem for generalized smooth functions.
\begin{defn}
\label{def:GCk}Let $U\subseteq\rti^{n}$, $Y\subseteq\rti^{d}$,
and $k=0$, $1$. Then, we say that $f:U\ra Y$ is a generalized $\mathcal{C}^{k}$-function
($f\in\gsfk k(U,Y)$), if there exists a net $f_{\eps}\in\mathcal{C}^{k}(U_{\eps},\R^{d})$
defining $f$ in the sense that:
\begin{enumerate}
\item \label{enu:domainGSF}$U\subseteq\left\langle U_{\eps}\right\rangle $;
\item \label{enu:equalityGSF}$f\left(\left[x_{\eps}\right]\right)=\left[f_{\eps}(x_{\eps})\right]\in Y$
for all $x=\left[x_{\eps}\right]\in U$;
\item \label{enu:moderateGSF}$\left(\partial^{\alpha}f_{\eps}(x_{\eps})\right)$
is $\rho$-moderate for all $x=[x_{\eps}]\in U$ and all multi-index
$\alpha\in\N^{n}$ such that $|\alpha|\leq k$;
\item \label{enu:GCksharpCont}For all multi-index $\alpha\in\N^{n}$ with
$|\alpha|=k$, the map $[x_{\eps}]\in U\mapsto\left[\partial^{\alpha}f_{\eps}(x_{\eps})\right]\in Y$
is well-defined and sharply continuous.
\end{enumerate}
\end{defn}

\noindent Actually, this definition works for any $k\in\N\cup\{+\infty]$,
and taking $k=+\infty$, property \ref{enu:GCksharpCont} can be proved
using moderateness of every derivative \ref{enu:moderateGSF} , see
Def.~\ref{def:netDefMap}, Thm.~\ref{thm:propGSF}, but also \cite[Thm. 16]{GKV19}
and \cite[Thm.~17]{GKV19} for the details.

Note, however, that the absolute value function $|-|:\rti\ra\rti$
is \emph{not} a $\gsfk1$-function because its derivative is not sharply
continuous at the origin, even though its right and left derivatives
exist; clearly, it is a $\gsfk0$-function. On the other hand, the
embedding theorem \cite[Thm.~25]{GKV19} (or any other regularization
process with a smooth mollifier we are interested to consider) yields
a GSF $\text{abs}(-)\in\gsf(\rti,\rti)$ as an embedding of $x\in\R\mapsto|x|\in\R$.
However, the regularization process results in $\text{abs}(x)\approx|x|$
and $\text{abs}(x)=|x|$ only if $|x|\ge r$ for some $r\in\R_{>0}$.
On the other hand, it is easy to prove that $x\in X\mapsto|f(x)|\in\rti$
is always sharply continuous if $f\in\gsf(X,\rti^{d})$.
\begin{thm}[Brouwer]
Every generalized continuous map $f\in\gsfk0([0,1]^{d},[0,1]^{d})$
has a fixed point.
\end{thm}

\begin{proof}
Let $f_{\eps}\in\mathcal{C}^{0}(U_{\eps},\R^{d})$ be continuous representatives
of $f$, so that $\left\langle U_{\eps}\right\rangle \supseteq[0,1]^{d}=[[0,1]_{\mathbb{R}}^{d}]\subseteq\rcrho^{d}$.
Classically proceeding by contradiction (see also \cite[Lem.~11]{GKV19}),
we have $\Omega_{\eps}\supseteq[0,1]_{\mathbb{R}}^{d}$ for all $\eps$
small. This means that $f_{\eps}$ , which is continuous, can be restricted
to $f_{\eps}|_{[0,1]^{d}}\ra\mathbb{R}^{d}$ and is continuous. Now,
the case might occur, that one component of $f_{\eps}^{j}$, for $j=1,2,...,d$
can be greater than $1$ or less than $0$ up to a negligible amount.
To deal with this case, we can take new representatives by setting
$\bar{f}_{\eps}$ as $\bar{f}_{\eps}^{j}=\min(f_{\eps}^{j},1)$.

\noindent We need to show that this function is continuous, but this
follows similarly to the standard case where we can define the $\min$
as

\[
\bar{f}_{\eps}^{j}=\min(f_{\eps}^{j},1)=-\max(-f_{\eps}^{j},-1)=-(\dfrac{1}{2}(-f_{\eps}^{j}-1+\lvert-f_{\eps}^{j}+1\rvert)).
\]

\noindent The last expression is a composition of continuous functions,
hence $\bar{f}_{\eps}=\min(f_{\eps}^{j},1)$ is continuous too. The
similar procedure works considering $\max(f_{\eps}^{j},1).$ This
yields that the initial function can be written as $f(x)=[(\bar{f_{\eps}^{1}}(x_{\eps}),\dots,\bar{f_{\eps}^{d}}(x_{\eps})]$,
i.e.~we can change representatives so that $f_{\eps}$ remains below
1 and above 0.

We therefore proved that we can find new representatives such that

\[
\forall^{0}\eps:\ \bar{f}_{\eps}\in\mathcal{C}^{0}([0,1]_{\mathbb{R}}^{d},[0,1]_{\mathbb{R}}^{d}).
\]

\noindent We can hence apply the standard version of Brouwer's Fixed
point theorem to find a fixed point $x_{\eps}\in[0,1]_{\mathbb{R}}^{d}$:
\[
\bar{f_{\eps}}(x_{\eps})=x_{\eps}.
\]

\noindent Setting $x:=[x_{\eps}]\in[0,1]^{d}$ finishes the proof.
\end{proof}
\noindent Like in the classical case, via translation this result
can be extended to any interval $[a,b]$, $a<b\in\rcrho$.

\section{\label{sec:Newton's-method}Newton-Raphson method in $d$-dimensions}

In this section, we consider Newton's method for GSF, its reduction
to the previous Banach fixed point theorem, as formulated by \cite{BIA},
and its quadratic convergence. We clearly always assume that the dimension
$d\ge1$.
\begin{defn}
\label{def:NRP}Let $U\subseteq\rcrho^{d}$ and $f\in\Gcinf(U,\rcrho^{d})$.
We say that $(x_{n})_{n\in\N}$ is \emph{a Newton(-Raphson) process
for $f$} if for all $n\in\N$ we have:
\begin{enumerate}
\item $x_{n}\in U$;
\item $\diff f(x_{n})$ is invertible;
\item $x_{n+1}=x_{n}-[\diff f(x_{n})]^{-1}f(x_{n})$.
\end{enumerate}
\end{defn}

\noindent On the basis of Thm.~\ref{thm:FR-forGSF}, we can still
say that $x_{n+1}$ is the point where the tangent line (if $d=1$)
at $f(x_{n})$ intersects the $x$-axis. Note, however, that $x_{n+1}=[x_{n+1,\eps}]\in\rcrho$
and hence this tangent line can be represented as an $\eps$-moving
ordinary tangent line at $f(x_{n})=[f_{\eps}(x_{n\eps})]$, which
could also assume infinite or infinitesimal values.

In order to derive the correctness of Newton's method from the Banach
contraction principle Thm.~\ref{thm:BFPT}, we define $g$ as:

\begin{equation}
g(x):=x-[\diff f(x)]^{-1}f(x)\label{eq:reductionNewtonToBanach}
\end{equation}
for all $x\in U$ such that $\diff f(x)$ is invertible. Note that
$x^{*}$ is a fixed point of $g$ if and only if it is a root of $f$.

Concerning invertibility of the differential, we have the following
\begin{lem}
Let $U$ be a sharply open set in $\rcrho$, $f\in\Gcinf(U,\rcrho)$
and $u\in U$ such that $\diff f(u)$ is invertible. Then there exists
a neighbourhood $V$ of $u$ such that $\diff f(v)$ is invertible
for all $v\in V$.
\end{lem}

\begin{proof}
Since $\diff f(u)$ is invertible, we have $|\diff f(u)|>0$ by $|\diff f(u)|=[|\diff f_{\eps}(u)|]$
and Lem.~\ref{lem:mayer}. This implies $|\diff f(u)|>\diff\rho^{q}$
for some $q\in\mathbb{N}$, once again from Lem.~\ref{lem:mayer}.
Thanks to Thm.~\ref{thm:FR-forGSF}, we have that $\diff f$ is also
a GSF; thereby, sharp continuity Thm.~\ref{thm:propGSF}.\ref{enu:GSF-cont}
yields $\lvert\diff f(v)-\diff f(u)\rvert<\diff\rho^{2q}$ for $v$
sufficiently near to $u$, let us say $v\in B_{\diff\rho^{Q}}(u)$.
Using the last property of norm listed at the beginning of Sec.~\ref{subsec:Topologies},
we have 
\[
\lvert\lvert\diff f(v)\rvert-\lvert\diff f(u)|\rvert\le\lvert\diff f(v)-\diff f(u)\rvert<\diff\rho^{2q}.
\]

\noindent This means, that 
\[
\diff\rho^{3q}\le\diff\rho^{q}-\diff\rho^{2q}<\lvert\diff f(u)|-\diff\rho^{2q}<\lvert\diff f(v)\rvert
\]
for all $v\in B_{\diff\rho^{Q}}(u)$, which is the conclusion by Lem.~\ref{lem:mayer}.
\end{proof}
It is well-known that, in general, Newton-Raphson iterations do not
converge, unless suitable conditions on the function $f$ are satisfied.
For example, even the smooth case $f(x):=27x^{3}-3x+1$ has only one
negative root $x^{*}\simeq-0.44$, but starting from $x_{0}=0$, we
have $x_{1}=1/3$ and $x_{n+1}>x_{n}>0$, e.g.~$x_{2}=1/6$, $x_{3}=1$,
$x_{4}\simeq0.68$, etc. In this case, the tangent line at the initial
point $x_{0}=0$ has negative slope, and therefore intersects the
$x$-axis at a point that is farther away from the negative root.
On the other hand, even in this case, the method converges if the
initial point $x_{0}$ is ``sufficiently near to the root $x^{*}$'',
e.g.~for $x_{0}<-\sqrt{3}/9$, where $f'(x_{0})>0$.

For these reasons, in literature we can find several sufficient conditions
for the convergence of Newton-Raphson iterations, see e.g.~\cite{Kan1,Kan2,Os1,Den1,Ort1}

The next result is inspired, in the setting of GSF, by \cite[Thm.~1]{BIA}.
\begin{thm}[Newton-Raphson, Ben-Israel]
\label{thm:NBI}Let $x_{0}\in\rti^{d}$, $r\in\rti_{>0}$, and $f\in\gsf(B_{r}(x_{0}),\rcrho^{d})$
be such that $\diff f(u)$ is invertible for all $u\in B_{r}(x_{0})$.
Let $M$, $N$, $k\in\rcrho_{>0}$ be such that for all $u$, $v\in B_{r}(x_{0})$,
we have: 
\begin{align}
|\diff f(v)(u-v)-f(u)+f(v)| & \le M|u-v|\label{eq:8BI}\\
|(\diff f(v)^{-1}-\diff f(u)^{-1})f(u)| & \le N|u-v|\label{eq:9BI}
\end{align}
and
\begin{align}
M|\diff f(x)^{-1}|+N & \le k\quad\forall x\in B_{r}(x_{0})\label{eq:10BI}\\
\lim_{n\to+\infty}k^{n} & =0.\label{eq:10bisBI}
\end{align}

\begin{equation}
|\diff f(x_{0})^{-1}||f(x_{0})|\le(1-k)r.\label{eq:Jacob}
\end{equation}
Then the sequence
\begin{equation}
x_{n+1}=x_{n}-(\diff f(x_{n}))^{-1}f(x_{n}),\quad\forall n\in\mathbb{N}\label{eq:withJacob}
\end{equation}
converges to a solution $x^{*}\in B_{r}(x_{0})$ of $f(x^{*})=0$.
\end{thm}

\begin{proof}
Let $g:\overline{B_{r}(x_{0})}\ra\rti^{d}$ be defined by \eqref{eq:reductionNewtonToBanach},
so that \eqref{eq:withJacob} becomes $x_{n+1}=g(x_{n})$. We prove
by induction that $x_{n}\in B_{r}(x_{0})$ for all $n\ge1$, the step
for $n=1$ following directly from \eqref{eq:Jacob} and \eqref{eq:10bisBI}.
Assume that $x_{i}\in B_{r}(x_{0})$ is true for all $1\le i\le n$,
we prove it for $n+1$:

\begin{align*}
x_{i+1}-x_{i} & =x_{i}-x_{i-1}-\diff f(x_{i})^{-1}f(x_{i})+\diff f(x_{i-1})^{-1}f(x_{i-1})\\
 & =\diff f(x_{i-1})^{-1}\diff f(x_{i-1})(x_{i}-x_{i-1})-\diff f(x_{i})^{-1}f(x_{i})+\diff f(x_{i-1})^{-1}f(x_{i-1})\\
 & =\diff f(x_{i-1})^{-1}\bigg(\diff f(x_{i-1})(x_{i}-x_{i-1})-f(x_{i})+f(x_{i-1})\bigg)\\
 & \phantom{=}+\bigg(\diff f(x_{i-1})^{-1}-\diff f(x_{i})^{-1}\bigg)f(x_{i}).
\end{align*}
Using \eqref{eq:8BI} and \eqref{eq:9BI} and the inductive assumption
$x_{i}$, $x_{i-1}\in B_{r}(x_{0})$, we obtain
\begin{align}
\left|x_{i+1}-x_{i}\right| & \le\left(M|\diff f(x)^{-1}|+N\right)|x_{i}-x_{i-1}|\nonumber \\
 & \le k|x_{i}-x_{i-1}|,\label{eq:contr}
\end{align}
where we also used \eqref{eq:10BI}. This implies

\[
|x_{i+1}-x_{0}|\le\sum_{j=1}^{i}k^{j}|x_{i}-x_{0}|=\frac{k(1-k^{i})}{1-k}|x_{1}-x_{0}|.
\]
Therefore, from \emph{\eqref{eq:Jacob} }we get\emph{ $|x_{1}-x_{0}|=|\diff f(x_{0})^{-1}||f(x_{0})|\le(1-k)r$
}and hence

\[
|x_{i+1}-x_{0}|\le kr(1-k^{j})<r.
\]
Since $x_{n+1}=g(x_{n})$, inequality \eqref{eq:contr} implies that
$g$ is a contraction on the orbit starting from $x_{0}\in\overline{B_{r}(x_{0})}$,
and the conclusion follows from \eqref{eq:10bisBI} and Thm.~\ref{thm:BFPT}\emph{.}
\end{proof}
From Taylor's theorem \ref{thm:Taylor}, we have
\[
\left|\diff f(v)(u-v)-f(u)+f(v)\right|\le\frac{1}{2}\left|\diff f^{2}(\xi)\right|\left|u-v\right|{}^{2}.
\]
Therefore, assumption \eqref{eq:8BI} surely holds for a small constant
$M$ if $u-v$ is sufficiently close to $0$, i.e.~for a sufficiently
small $r\in\rti_{>0}$ if $u$, $v\in B_{r}(x_{0})$.

We also recall that the inverse function theorem holds for GSF, see
\cite{GiKu16}:
\begin{thm}
\label{thm:localIFTSharp}Let $X\sse\rti^{d}$, let $f\in\gsf(X,\rti^{d})$
and suppose that for some $x_{0}$ in the sharp interior of $X$,
$\diff f(x_{0})$ is invertible in $L(\rti^{d},\rti^{d})$. Then there
exists a sharp neighborhood $U\sse X$ of $x_{0}$ and a sharp neighborhood
$V$ of $f(x_{0})$ such that $f:U\to V$ is invertible and $f^{-1}\in\gsf(V,U)$.
Moreover, from the chain formula we have $\diff(f^{-1})(v)=\left[\diff f(f^{-1}(v))\right]^{-1}$
for all $v\in V$.
\end{thm}

\noindent Applying Taylor's formula to $f^{-1}$, we can easily see
that $|(\diff f(v)^{-1}-\diff f(u)^{-1})f(u)|\le H\left|f(u)\right|\left|f(v)-f(u)\right|$.
The term $\left|f(v)-f(u)\right|$ is of the order $|v-u|$, and hence
condition \eqref{eq:9BI} holds for a small constant $N$ if $|f(u)|$
is sufficiently small, and this formalizes the intuitive idea that
``the initial iteration $x_{0}$ must be sufficiently close to the
root $x^{*}\in B_{r}(x_{0})$, so that $\left|f(u)\right|$ is sufficiently
small''.

\subsection{Quadratic convergence of Newton method}
\begin{defn}
We say that the sequence $(x_{n})_{n\in\mathbb{N}}\in\rcrho$ \emph{converges
at least of order} $q$ to $x^{*}$ if $\exists\alpha\in\rcrho_{>0}\,\exists N\in\N\,\forall n\in\mathbb{N}_{\ge N}:\ |x_{n+1}-x^{*}|\le\alpha|x_{n}-x^{*}|^{q}$.
If $q=2$, then we say that the sequence converges at least quadratically.
\end{defn}

\begin{thm}
Let $U\subseteq\rti^{d}$ and $f\in\gsf(U,\rti^{d})$ be such that
at $x^{*}\in U$ the differential $\diff f(x^{*})$ is invertible.
Assume that $(x_{n})_{n\in\N}$ is a Newton process for $f$ which
converges to $x^{*}$ as $n\to\infty$. Then for all $M>\frac{1}{2}\left|\diff f(x^{*})^{-1}\diff f(x^{*})\right|$
and for $n\mathbb{\in N}$ sufficiently large, we have $|x_{n+1}-x^{*}|\le M|x_{n}-x^{*}|^{2}$,
i.e.~$(x_{n})_{n\in\N}$ converges quadratically to $x^{*}$.
\end{thm}

\begin{proof}
Set $x_{n}-x^{*}=:h_{n}$. Using Taylor's theorem \ref{thm:Taylor},
for some $\mu_{n}\in[x_{n},x^{*}]$ we have

\[
f(x_{n}-h_{n})=f(x_{n})-h_{n}\cdot\diff f(x_{n})+\frac{h_{n}^{2}}{2}\diff f(\mu_{n}).
\]

\noindent Now, since $x_{n}-h_{n}=x^{*}$ and the Newton process $(x_{n})_{n\in\N}$
converges to $x^{*}$, we necessarily have $f(x^{*})=0$.

\[
f(x_{n})-(x_{n}-x^{*})\diff f(x_{n})+\frac{h_{n}^{2}}{2}\diff f(\mu_{n})=0.
\]

\noindent By Def.~\ref{def:NRP}, $\diff f(x_{n})$ is invertible,
hence

\[
\diff f(x_{n})^{-1}f(x_{n})-(x_{n}-x^{*})+\frac{h_{n}^{2}}{2}\diff f(x_{n})^{-1}\diff f(\mu_{n})=0
\]

\noindent From Def.~\ref{def:NRP} of Newton's process, we have

\[
x_{n+1}-x^{*}=\frac{h_{n}^{2}}{2}\diff f(x_{n})^{-1}\diff f(\mu_{n})
\]

\noindent which implies

\[
|x_{n+1}-x^{*}|=\frac{1}{2}\left|\diff f(x_{n})^{-1}\diff f(\mu_{n})\right||x_{n}-x^{*}|^{2}
\]

\noindent By sharp continuity Thm.~\ref{thm:propGSF} and the Fermat-Reyes
Thm.~\ref{thm:FR-forGSF}, $\diff f(x_{n})$ converges to $\diff f(x^{*})$
and $\mu_{n}\in[x_{n},x^{*}]$ converges to $x^{*}$, and for sufficiently
large $n$, $M>\frac{1}{2}\left|\diff f(x_{n})^{-1}\diff f(\mu_{n})\right|\to\frac{1}{2}\left|\diff f(x^{*})^{-1}\diff f(x^{*})\right|$
and this proves the claim.
\end{proof}

\section{\label{sec:Examples}Examples}

In a non-Archimedean setting such as $\rti$, the calculations needed
to show that assumptions of Thm.~\ref{thm:NBI} hold in a given example
become immediately very involved. We therefore sometimes used the
algebraic symbolic algorithms of Wolfram Mathematica ver.~14.3 to
arrive at the searched solutions.

A simple, but ordinary smooth case without singularities, that can
be verified without using a computer is the following:
\begin{example}
Set $f(u):=1-u^{2}$. We want to show that assumptions \eqref{eq:8BI}-\eqref{eq:Jacob}
of Thm.~\ref{thm:NBI} are satisfied. We have $\diff f(u)=-2u$ and
hence $(\diff f(u))^{-1}=-\frac{1}{2u}$. We have to find suitable
constants $M$, $N\in\rti_{>0}$ and using those, a related constant
$k\in\rti_{>0}$ such that the properties hold for our choice of initial
point $x_{0}$ and radius $r\in\rti_{>0}$. This leads to the following
computations:\medskip{}

\noindent\textbf{Assumption~\eqref{eq:8BI}:} 
\begin{align*}
|\diff f(v)(u-v)-f(u)+f(v)| & =|-2v(u-v)-(1-u^{2})+1-v^{2}|\\
 & =|-2uv+2v^{2}+u^{2}-v^{2}|\\
 & =|u^{2}-2uv+v^{2}|=|(u-v)^{2}|\leq M|u-v|.
\end{align*}
This holds if $|u-v|\leq M$ and hence $M:=2r$ is a suitable choice
to satisfy assumption \eqref{eq:8BI}.\medskip{}

\noindent\textbf{Assumption~\eqref{eq:9BI}:}
\begin{align*}
|(\diff f(v)^{-1}-\diff f(u)^{-1})f(u)| & =\left|(-\frac{1}{2v}+\frac{1}{2u})(1-u^{2})\right|=\left|\frac{-u+v}{2uv}\right|\left|1-u^{2}\right|\\
 & =|u-v|\left|\frac{1}{2uv}\right|\left|1-u^{2}\right|\leq N|u-v|.
\end{align*}
Therefore, we necessarily must have $N\geq\frac{\left|1-u^{2}\right|}{|2uv|}$.
Knowing that $u$, $v\geq x_{0}+r\geq0$, we can set $N:=\frac{1-(x_{0}+r)^{2}}{2(x_{0}-r)^{2}}$
so that \eqref{eq:9BI} holds as well.\medskip{}

\noindent\textbf{Assumption~\eqref{eq:10BI}:} We have $M|\diff f(x)^{-1}|+N=2r\frac{1}{2(x_{0}-r)^{2}}+\frac{1-(x_{0}+r)^{2}}{2(x_{0}-r)^{2}}\leq k,$
so we can take $k$ equal to the left hand side of the inequality
to ensure \eqref{eq:10BI}.\medskip{}

\noindent\textbf{Assumption~\eqref{eq:10bisBI}:} Now, we have to
choose $x_{0}\approx1$ and $r\approx0$. Set $x_{0}:=1-\diff\rho^{2}$
and $r:=\diff\rho$. To verify that $\lim_{n\to+\infty}k^{n}=0$,
we have to show that there exists $R\in\R_{>0}$ such that $k\leq\diff\rho^{R}.$
Plugging our chosen values into the form of $k,$ we find
\[
k=\frac{2\diff\rho+1-((1-\diff\rho^{2})+\diff\rho)^{2}}{2((1-\diff\rho^{2})-\diff\rho)^{2}}=\frac{\diff\rho^{2}+2\diff\rho^{3}-\diff\rho^{4}}{2(1-\diff\rho-\diff\rho^{2})^{2}}\leq\diff\rho^{R},
\]
which holds e.g.~for $R:=1$.\medskip{}

\noindent\textbf{Assumption~\eqref{eq:Jacob}:
\begin{align*}
|\diff f(x_{0})^{-1}||f(x_{0})| & =\frac{|1-(1-\diff\rho^{2})^{2}|}{2|1-\diff\rho^{2}|}=\frac{1-(1-2\diff\rho^{2}+\diff\rho^{4})}{2(1-\diff\rho^{2})}=\frac{2\diff\rho^{2}-\diff\rho^{4}}{2(1-\diff\rho^{2})}\\
 & \leq(1-k)\diff\rho=\diff\rho-\diff\rho\frac{\diff\rho^{2}+2\diff\rho^{3}-\diff\rho^{4}}{2(1-\diff\rho-\diff\rho^{2})^{2}}.
\end{align*}
}This inequality holds by comparing the order of the infinitesimals
numbers appearing on the two sides.
\end{example}

The following is a non-Archimedean version of the previous one.
\begin{example}
Let $f(u):=Ha^{2}-Hu^{2}$ where $a$, $H\in\rcrho_{>0}$, $a$ is
an infinitesimal number, but $Ha^{2}\in\rcrho$ is infinite. First,
we notice that $\diff f(u)=-2Hu$ and $(\diff f(u))^{-1}=-\frac{1}{2Hu}$.\medskip{}

\noindent\textbf{Assumption~\eqref{eq:8BI}:
\begin{align*}
|\diff f(v)(u-v)-f(u)+f(v)| & =\left|(-2Hv)(u-v)-Ha^{2}+Hu^{2}+Ha^{2}-Hv^{2}\right|\\
 & =\left|-2Hvu+2Hv^{2}+Hu^{2}-Hv^{2}\right|\\
 & =\left|H(u-v)^{2}\right|\leq M|u-v|.
\end{align*}
}This implies that $|H||u-v|\leq M$ and hence $M:=2r|H|$ is a sufficient
choice to satisfy this inequality.\medskip{}

\noindent\textbf{Assumption~\eqref{eq:9BI}:
\begin{align*}
|(\diff f(v)^{-1}-\diff f(u)^{-1})f(u)| & =\left|(\frac{-1}{2Hv}+\frac{1}{2Hu})(Ha^{2}-Hu^{2})\right|\\
 & =\left|-\frac{Ha^{2}}{2Hv}+\frac{Hu^{2}}{2Hv}+\frac{Ha^{2}}{2Hu}-\frac{Hu^{2}}{2Hu}\right|\\
 & =\left|\frac{u^{2}-a^{2}}{2v}+\frac{a^{2}-u^{2}}{2u}\right|\\
 & =\left|\frac{u^{3}-a^{2}u+a^{2}v-vu^{2}}{2uv}\right|\\
 & =\left|\frac{(u-v)(u^{2}-a^{2})}{2uv}\right|\leq N|u-v|.
\end{align*}
}So we must necessarily have $\left|\frac{u^{2}-a^{2}}{2uv}\right|\leq N$
and we can choose $N:=\frac{a^{2}-(x_{0}+r)^{2}}{2(x_{0}-r)^{2}}$.\medskip{}

\noindent\textbf{Assumption~\eqref{eq:10BI}:} $M|\diff f(x)^{-1}|+N=2r|H|\left|\frac{1}{2Hu}\right|+N=\frac{2r}{2|u|}+\frac{a^{2}-(x_{0}+r)^{2}}{2(x_{0}-r)^{2}}\leq\frac{2r+a^{2}-(x_{0}+r)^{2}}{2(x_{0}-r)^{2}}=:k$.\medskip{}

\noindent\textbf{Assumption~\eqref{eq:10bisBI}:} To find the constant
$k$, we can choose $x_{0}:=a+a^{2}\diff\rho^{R+2}$ and $r:=a^{2}\diff\rho^{R+1}$
for some $R\in\mathbb{R}_{>0}.$ These choices lead us to the following
$k$
\[
k=\frac{\diff\rho^{R+1}\left(2-2a-2a\diff\rho-a^{2}\diff\rho^{R+1}-2a^{2}\diff\rho^{R+1}-a^{2}\diff\rho^{R+3}\right)}{2(1+a\diff\rho^{R+2}-a\diff\rho^{R+1})^{2}}.
\]
Now, we observe that $k\leq\diff\rho^{R}$ holds for our choice of
$x_{0}$ and $r$, and for any positive $R$ and this proves the fourth
inequality.\medskip{}

\noindent\textbf{Assumption~\eqref{eq:Jacob}:
\begin{align*}
|\diff f(x_{0})^{-1}||f(x_{0})| & =\left|\frac{1}{2Hx_{0}}\right|\left|Ha^{2}-Hx_{0}^{2}\right|=\frac{|a^{2}-x_{0}^{2}|}{2|x_{0}|}\\
 & =\left|\frac{a^{2}-\left(a^{2}+2a^{3}\diff\rho^{R+2}+a^{4}\diff\rho^{2R+4}\right)}{2\left(a+a^{2}\diff\rho^{R+2}\right)}\right|\\
 & =\frac{2a^{3}\diff\rho^{R+2}+a^{4}\diff\rho^{2R+4}}{2\left(a+a^{2}\diff\rho^{R+2}\right)}=\frac{a^{2}\diff\rho^{R+2}\left(2+a\diff\rho^{R+2}\right)}{2\left(1+a\diff\rho^{R+2}\right)}\\
 & \leq(1-k)a^{2}\diff\rho^{R+1}.
\end{align*}
}We have that $\frac{a^{2}\diff\rho^{R+2}\left(2+a\diff\rho^{R+2}\right)}{2\left(1+a\diff\rho^{R+2}\right)}=\diff\rho^{R+1}a^{2}\frac{\diff\rho\left(2+a\diff\rho^{R+2}\right)}{2\left(1+a\diff\rho^{R+1}\right)}$
and hence $\frac{\diff\rho\left(2+a\diff\rho^{R+2}\right)}{2\left(1+a\diff\rho^{R+1}\right)}\leq1-k$
which holds since $\diff\rho$ is an infinitesimal number, hence also
proving the last inequality.
\end{example}

For the last example, we consider a GSF defined by regularizing the
ramp function $\text{ramp}(x):=\max(0,x)$. In order to simplify the
calculations, we consider the gauge $\rho_{\eps}=\eps$ and the mollification
with $\mu(x):=\max(0,1-x^{2})\frac{3}{4}$, $\mu_{\eps}(x):=\frac{1}{\eps}\mu(\frac{x}{\eps})$
(which is clearly not a Colombeau mollifier). For some calculations
related to this example, we used Wolfram Mathematica ver.~14.3.
\begin{example}
Taking the convolution of the aforementioned ramp function with $\mu_{\eps}$
we obtain $f(x):=\left[\left(\text{ramp}*\mu_{\eps}\right)(x)\right]=\frac{3}{4\diff\rho}\frac{(3\diff\rho-x)(\diff\rho+x)^{3}}{12\diff\rho^{2}}\in\rti$
for all $x\in\rti$, and hence $\diff f(x)=\left[\left(\text{ramp}*\diff\mu_{\eps}\right)(x)\right]=\frac{(2\diff\rho-x)(\diff\rho+x)^{2}}{4\diff\rho^{3}}$,
$\diff f(x)^{-1}=\frac{4\diff\rho^{3}}{(2\diff\rho-x)(\diff\rho+x)^{2}}$.
With these expressions we obtain the following results for the assumptions
of Thm.~\ref{thm:NBI}:

\medskip{}

\noindent\textbf{Assumption~\eqref{eq:8BI}:} The inequality $|\diff f(v)(u-v)-f(u)+f(v)|\leq M|u-v|$
leads to
\begin{align*}
M & \geq\frac{1}{192\diff\rho^{3}}\left(\left|16\diff\rho u^{3}+16\diff\rho u^{2}v+16\diff\rho uv^{2}+(16\diff\rho-3)v^{3}\right.\right.\\
 & \phantom{\ge}\left.\left.-96\diff\rho^{3}u+(9\diff\rho^{2}-96\diff\rho^{3})v-(128\diff\rho^{4}-6\diff\rho^{3})\right|\right).
\end{align*}
Now, to choose an $M$ that satisfies this inequality, we use $u$,
$v\leq x_{0}+r$, and plugging this into the last inequality we find
that the following
\[
M:=\frac{|64\diff\rho-3|}{192\diff\rho^{3}}\left|\left(x_{0}+r+\diff\rho\right)^{2}(x_{0}+r-2\diff\rho)\right|
\]
satisfy the claim. We recall that in this example, $\diff\rho=[\eps]\in\rti=\Rtil$
is the usual Colombeau ring of generalized numbers.\medskip{}

\noindent\textbf{Assumption~\eqref{eq:9BI}:} Applying the same
procedure to $|(\diff f(v)^{-1}-\diff f(u)^{-1})f(u)|\leq N|u-v|$,
we find
\[
N:=\frac{16\diff\rho(3\diff\rho-x_{0}-r)(\diff\rho-x_{0}-r)(\diff\rho+x_{0}+r)^{2}}{(2\diff\rho-x_{0}+r)^{2}(\diff\rho+x_{0}-r)^{2}}.
\]
\medskip{}

\noindent\textbf{Assumption~\eqref{eq:10BI}:} For $k$, the calculations
yield
\begin{align*}
M|\diff f(x)^{-1}|+N & =\frac{|62\diff\rho-3|}{3}\frac{|(x_{0}+r)+\diff\rho)^{2}(x_{0}+r-2\diff\rho)|}{|(2\diff\rho-x_{0}+r)(x_{0}-r+\diff\rho)^{2}|}\\
 & \phantom{=}\frac{16\diff\rho(3\diff\rho-x_{0}-r)(\diff\rho-x_{0}-r)(\diff\rho+x_{0}+r)^{2}}{(2\diff\rho-x_{0}+r)^{2}(\diff\rho+x_{0}-r)^{2}}\\
 & =\frac{(x_{0}+r+\diff\rho)^{2}}{3(2\diff\rho-x_{0}+r)^{2}(x_{0}-r+\diff\rho)^{2}}\left(|64\diff\rho-3|\left|(x_{0}+r-2\diff\rho)\right.\right.\\
 & \phantom{=}\left.(2\diff\rho-x_{0}+r)\right|\left.+48\diff\rho|(3\diff\rho-x_{0}-r)(\diff\rho-x_{0}-r)|\right)=:k.
\end{align*}
\medskip{}

\noindent\textbf{Assumption~\eqref{eq:10bisBI}:} For this example,
a suitable choice of the initial point and the radius would be $x_{0}:=-\diff\rho+\diff\rho^{2}$
and $r=\diff\rho+\sqrt{\diff\rho}.$ We again used Mathematica to
confirm that $k\leq\diff\rho^{R}$ for these choices, in fact, a suitable
$R$ would be $R=10^{-8}$.\medskip{}

\noindent\textbf{Assumption~\eqref{eq:Jacob}:} The test of $|\diff f(x_{0})^{-1}||f(x_{0})|\leq(1-k)r$
was also checked with Mathematica and the choices of the initial point
and the radius also satisfy this inequality.

\noindent As a further verification in this example, using Mathematica
all the inequalities \eqref{eq:8BI}-\eqref{eq:Jacob} have also been
tested to hold for $\eps$ sufficiently small and the previous choices
of the constants $M=[M_{\eps}]$, $N=[N_{\eps}]$, $k=[k_{\eps}]$,
$x_{0}=[x_{0\eps}]$ and $r=[r_{\eps}]$.
\end{example}

\section{Conclusions}

The present article is only one of a series \cite{NuGi1,BIG,GGBL,KeGi24b,LeLuBaGi17,GiKu16,GK15,GKV15,GKV19,MTG}
aiming to show that GSF, in spite of fact that they embed ordinary
Schwartz distributions, share with ordinary smooth functions several
non trivial properties. As proved in \cite{BIG,GGBL,LeLuBaGi17,MTG},
this allow the creation of rigorous models for real-world applications
showing singularities. The calculations presented in these models
faithfully correspond to the informal intuitive formal manipulations
frequently performed by applied researchers. The final examples presented
in this paper open the possibility to use computer aided symbolic
calculations to help in the numerical calculus of these generalized
functions.

\end{document}